\documentclass[11pt, leqno]{amsart}

\usepackage{setspace}
\usepackage{graphicx}
\usepackage{amsfonts,amsmath,oldgerm,amssymb,amscd,bbm}
\usepackage{dsfont}
\usepackage[T1]{fontenc}

\newtheorem{theorem}{Theorem}[section]
\newtheorem{lemma}[theorem]{Lemma}
\newtheorem{conjecture}{Conjecture}

\newtheorem{remark}[theorem]{Remark}

\newcommand{\C}{\mbox{$\mathbb C$}}
\newcommand{\Q}{\mbox{$\mathbb Q$}}
\newcommand{\A}{\mbox{$\mathbb A$}}
\newcommand{\ol}{\overline}

\newcommand{\nind}{\noindent}
\author{Indranil Biswas, R.V. Gurjar and Sagar U. Kolte}

\begin{document}

\title{On the Zariski-Lipman conjecture for normal algebraic surfaces}
\maketitle

\begin{small}
Abstract. We consider the Zariski-Lipman Conjecture on free module of derivations for 
algebraic surfaces. Using the theory of non-complete algebraic surfaces, and some basic results about
ruled surfaces, we will prove the conjecture for several classes of affine and projective surfaces.
\end{small}

\vspace{5mm}

\section*{Introduction}
\nind
The following is a well-known conjecture due to O. Zariski and J. Lipman:
\nind
\begin{conjecture}
Let $V$ be an algebraic variety over a field $k$ of characteristic 0, let $p$ be a (closed) point of $V$ 
and let $R$ be the local ring of $V$
at $p$. If the module of $k$-derivations $Der_k(R)$ is a free $R$-module then $V$ is smooth at $p$.\\

\end{conjecture}
\nind
Lipman proved in \cite{Li} that if $Der_k(R)$ is $R$-free then $R$ is normal. In view of this, an equivalent 
formulation of this conjecture (obtained by shrinking $V$) is the following:\\

\nind
{\it Assume that $V$ is normal and the tangent bundle of the smooth locus of $V$ is a trivial bundle. Then $V$ is smooth.}\\ 

\nind
To see the equivalence of the two statements we can assume that $V$ is affine. It is well-known that $Der_k(R)$ is a 
reflexive $R$-module (being the dual of $\Omega_k(R)$). Hence if $R$ is normal then any element of $Der_k(R)$ is 
determined by its restriction to the smooth locus of $V$. If the tangent bundle of the smooth locus of $V$ is trivial 
then a free basis of the module of cross-sections of the tangent bundle of the smooth locus of $V$ gives a free basis 
of $Der_k(R)$ as an $R$-module. Similarly, any free basis
of $Der_k(R)$ as an $R$-module gives a trivialization of the tangent bundle on the smooth locus of $V$, if $V$ is 
a sufficiently small Zariski-open neighborhood of $p$.\\ 

\nind
Without the normality assumption this conjecture is false even for dimension $V=1$. In \cite{H} M. Hochster proved 
the conjecture for positively graded domains over $k$. G. Scheja and U. Storch proved the conjecture for 
hypersurfaces in \cite{S}. A similar proof was given by the second author around 1975 (unpublished). We mention here 
that by a result of M. Artin any normal germ of a complex variety $(V,p)$ with an
isolated singularity at $p$ is algebraic. J. Becker \cite{becker} showed that the
Zariski-Lipman conjecture is true if it can be shown to be true when the
singularities of $V$ are isolated. J. Steenbrink and D. Van Straten
\cite{steenbrink} settled the
conjecture for isolated singularities of dimension three or more. H. Flenner,
\cite{flenner}, proved the conjecture for non-isolated singularities of dimension
three or more, assuming that the singularities are of codimension at least 3.
Recently, few more cases of the conjecture have been proved.

(i) R. K\"{a}llstr\"{o}m \cite{K} proved the conjecture for all complete intersections. 

(ii) S. Druel \cite{D} proved the conjecture for local rings with log canonical singularities.

(iii) Independently, P. Graf \cite{G} also proved (ii) by a different method.

As far as we can see none of these results imply our results.

Given a smooth algebraic surface $V$ and a projective completion $\ol V$
of it such that $D\, :=\, \ol{V} - V$ is a divisor whose only singularities are nodes, 
the \emph{logarithmic Kodaira dimension} of $V$, 
denoted by $\ol{\kappa}(V)$, is defined as the supremum of the dimensions of the images of $\ol{V}$ 
under the rational maps defined by $H^0(\ol{V}, n(K+D))$, $n \geq 1$. If this linear system is 
trivial for all 
$n\geq 1$ then we define $\ol{\kappa}(V)=-\infty$ \cite{I}.

Our approach about this conjecture in this paper is global and it uses the theory of non-complete algebraic surfaces
developed by S. Iitaka, Y. Kawamata, T. Fujita, M. Miyanishi, T. Sugie, S. Tsunoda and other Japanese mathematicians. 
This theory has proved to be very effective in the solution of many problems about non-complete algebraic surfaces, 
and the arguments in this paper is just one more instance of this. 
We also use some results from the theory of vector bundles on smooth projective curves. Although some of our 
arguments are valid
assuming only projectivity of the module of derivations, for many arguments we need the full force of the assumption 
that the tangent bundle of the smooth locus is trivial. As can be seen from the somewhat involved proofs in 
this paper this stronger hypothesis is justified by the difficulty of the general conjecture. In this paper 
we verify the conjecture for all affine surfaces $V$ such that $\ol \kappa(V-Sing(V)) \leq 1$,
and prove the conjecture in almost all the cases when $V$ is projective. In particular, we prove that if $V$ 
is projective, the tangent bundle of $V-Sing(V)$ is trivial and $\overline\kappa(V-Sing(V))=0$ or $1$ 
then $V$ is smooth.

All the varieties in this paper will be assumed to be over an algebraically closed field $k$ of 
characteristic 
$0$. If $V$ is an algebraic surface then by $V^0$ we denote the
smooth locus of $V$.\\
We now state the results of this paper.
\begin{theorem}
Let $V$ be an algebraic surface defined over $k$. Assume that the
tangent bundle of $V^0$ is trivial. The following statements hold:

\begin{enumerate}

 \item If $V$ is an affine algebraic surface such that $\ol \kappa(V^0)
       \leq 1$, then $V$ is smooth.\\

 \item If $V$ is a projective surface, then $\ol \kappa(V^0) \leq 0$ and $V$
       has at most one singularity.\\

 \item If $V$ is a projective surface such that $\ol \kappa(V^0)=0$, then $V$ is smooth.\\

% \item Let $V$ be a projective normal surface such that $\ol \kappa(V^0)= -
%       \infty$. Then $V$ has atmost one singularity.\\

 \item Assume that $V$ is a projective surface such that $\ol \kappa(V^0)= -
       \infty$. Let $p$ be the unique singular point of $V$. Then there
       exists a resolution of singularities $\pi : W \rightarrow V$ such
       that there is a $\mathbb{P}^1$-fibration $W \rightarrow C$, where $C$ is
       a smooth projective curve. If the genus of $C$ is at least 2, then
       $V$ is smooth.\\

 \item With the notation in (4), if $C$ is a rational curve then $V$ is smooth.\\ 
 
 \item With the notation in (4), let $C$ be an elliptic curve. If
       at least one singular fiber of $W \rightarrow C$ has a non-reduced
       feather (defined later), then $V$ is smooth.\\

\end{enumerate}

\end{theorem}
{\it Remark.} We have not been able settle the Zariski-Lipman Conjecture completely for 
projective case when $\ol\kappa(V^0)=-\infty$. If the exceptional divisor is a smooth elliptic curve then 
we have proved the conjecture using arguments from the theory of rank $2$ vector bundles. Similarly, if the
${\mathbb P}^1$-fibration has at least one non-reduced feather then we have proved the conjecture. But if the base 
of the ${\mathbb P}^1$-fibration is an elliptic curve and the exceptional divisor has many irreducible 
components, and only reduced feathers, then we have not been able to settle the conjecture. We believe that 
this case in rare and hope to be able to settle it in future.

{\it Acknowledgements.} The authors would like to thank the referee for reading the paper carefully and
making useful suggestions for improving the presentation.

\section{Preliminaries}

Let $X$ be a smooth projective surface. For a (possibly reducible) reduced curve $A$ on $X$, by 
a \emph{component}
 of $A$ we mean an irreducible component of $A$.
An irreducible smooth rational curve $A$ on $X$ with $A^2=-n$ is called a $(-n)$-curve. We will 
mainly use this terminology when $n>0$.\\

We also recall that a reduced effective divisor $D= D_1+D_2+\ldots +D_r$ on a smooth surface is called a
divisor with simple normal crossings (SNC, for short) if every component  $D_i$
is smooth and the only singularities of $D$ are simple nodes.\\

Let $f:X\rightarrow C$ be a morphism from a smooth surface $X$ onto a smooth curve $C$. Let 
$\sum_i a_iA_i$ be a scheme-theoretic fiber of $f$. Then G.C.D. $\{a_1,a_2,\ldots\}$ is called 
the {\it multiplicity} of the fiber. The integer $a_i$ is called the {\it multiplicity} of $A_i$ in the fiber.  

With $f,X,C$ as above, an irreducible curve $S\subset X$ is called a {\it cross-section} of $f$ if $f:S\rightarrow C$ 
is an isomorphism.\\

We will implicitly use the following easy consequence of Hodge Index Theorem (called Zariski's Lemma).\\

{\it Let $\pi:X\to B$ be a morphism on a smooth projective surface $X$ onto a smooth projective curve $B$ such
that a general fiber of $\pi$ is irreducible. Let $F_0:\Sigma a_iC_i$ be a singular fiber of $\pi$ (i.e. $F_0$
is scheme-theoretically not isomorphic to a general fiber of $\pi$). If $F_0$ is not irreducible then the intersection
form on a union of any proper subset of set of the components $\{C_i\}$ is negative definite. The intersection form on
$\{C_i\}$ has exactly one eigenvalue which is equal to $0$. In fact, the only rational divisors supported on $\{C_i\}$
with self-intersection $0$ are rational multiples of $F_0$.}\\

In our proofs some well-known properties of a singular fiber of a $\mathbb{P}^1$-fibration 
on a smooth projective surface will be implicitly used. We state this result for the sake of 
completeness. See \cite[Chapter I, \S~ 4.4.1]{M81}.

\begin{lemma} 
Let $f:X\rightarrow C$ be a $\mathbb{P}^1$-fibration on a smooth projective surface $X$ onto a 
smooth projective curve $C$. Let $F:=a_1A_1+\ldots+a_rA_r$ be a scheme-theoretic fiber of $f$, 
where $A_i$ are the components of $F$. Then we have:
\begin{enumerate}

  \item G.C.D. $(a_1,\ldots,a_r)=1$.
  \item The reduced divisor $F_{red}$ is SNC. Further, $F$ is a tree of smooth rational curves.
  \item At least one $A_i$ is a $(-1)$-curve. If $a_i=1$ for some $i$, then there is a $(-1)$-curve 
$A_j$ in $F$ such 
       that $j\neq i$. Any $(-1)$-curve in $F$ meets at most two other components of $F$.
\end{enumerate}
\end{lemma}

In connection with a $\mathbb{P}^1$-fibration we will use another notion, viz. that of 
a {\it feather}.    
Let $W$ be a smooth quasi-projective surface with an open embedding into a smooth projective surface $X$ such that
$D:=X-W$ is a connected SNC divisor. Suppose that there is a $\mathbb{P}^1$-fibration $f:X\rightarrow C$ which restricts 
to an $\A^1$-fibration on $W$. Then $D$ contains a unique component, say $A_1$, which meets every fiber of $f$ 
exactly once. This is because a general fiber of the morphism $f$ restricted to $W$ is isomorphic to $\A^1$ 
which has only one place at infinity and it is the point of intersection of $A_1$ and the closure of this $\A^1$.\\

\nind
Let $F$ be a singular fiber of $f$. Since $A_1$ is the only component of $D$ which does not lie in a fiber
of $f$, from connectedness of $D$ we see that $D\cap F$ is also connected.\\
Assume also that every component of $F$ intersects $D$ (this happens, for example, if $W$ is affine) and let $A_i$ be a 
component of $F$ which is not contained in $D$. We claim that $D\cap A_i$ is a single point. If $A_i$ intersected $D$ 
in more than one point then connectedness of $D\cap F$ would imply that $F$ is not a tree of $\mathbb{P}^1$'s. 
We call such a component $A_i$ of $F$ a {\it feather}.\\ 

We will state the following result which combines the works due to Hochster
\cite{H} (see also, \cite{P}), Scheja-Storch \cite{S} and K{\"a}llstr{\"o}m \cite{K}.

\begin{theorem} \mbox{}
\begin{enumerate}
 \item Let $V$ be an affine algebraic variety such that the algebraic tangent
bundle of
       $V^0$ is trivial. If the coordinate ring of $V$ is a positively graded
       domain, then $V$ is smooth.
  \item If $V$ is a complete-intersection singularity such that the tangent bundle of $V^0$ is trivial, 
       then $V$ is smooth. 

   \item If $V$ is a non-smooth algebraic surface with at most rational singularities, then
       the tangent bundle of $V^0$ is not trivial.
\end{enumerate}

\end{theorem}

To see part (3) of Theorem 1.2 note that if the algebraic tangent bundle of
$V^0$ is trivial and $V$ is a surface, then the singularities are Gorenstein.
Gorenstein rational singularities are rational double points. This implies that
the singularities are hypersurface singularities. By Theorem 1.2 (2) $V$ is smooth.

We will frequently reduce our problem to cases that can be answered by the
above theorem.

P. Wagreich \cite[\S~1.8]{W} defined the notion of an \emph{elliptic singularity}.\\
Let $p$ be a normal singularity of a two dimensional analytic space $V$
with $\mathcal{O}_{V,p}$ the local ring at $p$. Let $\pi : M \rightarrow V$
be the minimal resolution of $p$. Let $E=\pi^{-1}(p)$ be the exceptional
divisor. By $P_a(D)$ we denote the arithmetic genus of the divisor $D$, defined
as $(D^2+D \cdot K)/2+1$. Then the
arithmetic genus of $\mathcal{O}_{V,p}$, denoted by $P_a(\mathcal{O}_{V,p})$, is defined as $max~\{P_a(D)\}$, 
where $D$ ranges over all positive divisors whose support is contained in $E$. If the arithmetic genus is $1$ then 
the singularity is called an \emph{elliptic singularity}.\\

\nind
Let $V$ be a
smooth quasi-projective surface and let $X$ be a smooth projective completion of $V$
such that $D:=X-V$ is SNC.

We will need several notions from the theory of Zariski-Fujita decomposition of pseudo-effective divisors 
\cite[\S~6]{F}.

Recall that an integral divisor $\Delta$ on a smooth projective surface $X$ is \emph{pseudo-effective} if 
$\Delta\cdot H\geq 0$ for every ample divisor $H$ on $X$. Let $K$ denote the canonical divisor of $X$.
If $\ol\kappa(X-D)\geq 0$ then some multiple of $|K+D|$ is an effective divisor. Hence in this case $K+D$ is
pseudo-effective. This remark will be implicitly used in what follows.

Suppose that $K+D$ is pseudo-effective. Then there is a unique decomposition $K+D\approx P+N$, where 
$\approx$ denotes numerical equivalence, $P$ is a nef $\Q$-divisor, $N$ is an effective $\Q$-divisor such that the 
intersection form on the components of $N$ is negative definite, and $P\cdot D_i=0$ for any component $D_i$ in 
the support of $N$. 

Let $\ol S$ be a smooth complete surface and let $D$ be an effective reduced divisor on it. For any irreducible 
component $Y$ of $D$, we denote $Y \cdot  (D-Y)$ by $\beta(Y)$. This $\beta(Y)$
is called the \emph{branching number} of $Y$ in $D$.\\

$Y$ is called a \emph{tip} of $D$ if $\beta(Y)=1$. It is called a \emph{rational tip} if $Y = \mathbb{P}^1$. A sequence 
$C_1, \cdots , C_r$ of components of $D$ is called a \emph{rational twig} $T$ of $D$ if each $C_i$ is a rational 
normal curve, $\beta(C_1)=1$, $\beta(C_j)=2$ and $C_{j-1} \cdot C_{j}=1$ for $2 \geq j \geq r$.
The curve $C_1$ is called the 
tip of this twig $T$.\\

Since $\beta(C_r)=2$, there is a component $C$ of $D$, not in $T$, such that $C_r \cdot C=1$. If $C$ is a rational 
tip of $D$, then $T'=T+C$ is a connected component of $D$ and it will be called a \emph{rational club} of $D$.\\

When the above $C$ is isomorphic to $\mathbb{P}^1$ and $\beta(C)=2$, then $T'$ is a rational twig of $D$. Otherwise, $T$ 
is called a \emph{maximal rational twig} of $D$ and $C$ is called the \emph{branching component} of $T$.\\

If the intersection form on the components of $T$ is negative definite, then $T$ is said to be \emph{contractible}.
Assume that $T$ is a maximal rational contractible twig of $D$. The element $N \in \mathbb{Q}(T)$ such that 
$N \cdot C_1=-1$ and $N \cdot C_j=0$ for $j \geq 2$ is called the \emph{bark} of $T$. If $T'= C_1+
\ldots +C_r+C$ is a 
contractible rational club of $D$, the bark of $T'$ is defined to be the $\mathbb{Q}$-divisor $N'$ in
$\mathbb{Q}(T')$ 
such that $N' \cdot C_1= N' \cdot C= -1$ and $N' \cdot C_j=0$ for $2 \leq j \leq r$. For a connected component $Y$ 
of $D$ which is a rational normal curve, its bark is defined to be $2(-Y^2)^{-1}Y$.\\

Fujita has proved that if $\ol\kappa(V^0)\geq 0$ then all the rational twigs of $D$ are contractible 
\cite[Lemma 6.13]{F}. Let $Bk(D)$ denote the sum of all the barks of maximal rational contractible twigs of $D$. 
In view of Theorem 1.2 (3), 
we can assume that none of the singular points of V is a rational singular point. In particular, in Fujita's
terminology no singular point is an $``$abnormal rational club''. Hence we do not have to consider the 
$``$thicker bark'' of D.\\

Let $P+N$ be the Zariski-Fujita decomposition of $D$. We now state a result due to
Fujita, \cite[Lemma 8.7]{F}, about the divisor $D$ when $\ol \kappa(V^0)=0$. 

\begin{lemma}
Assume that $\ol\kappa(V^0)=0$. Assume also that any $(-1)$-curve in $D$ meets at least three other 
components of $D$. If $Bk(D)$ = $N$, then any connected component of $D$ is one
of the following
\begin{enumerate}

 \item A minimal resolution of a quotient singular point.
 \item A tree of $\mathbb{P}^1$'s with exactly two branch points such that the
branch points are connected by a (possibly empty) linear chain of
$\mathbb{P}^1$'s and each branch point meets exactly two other (-2)-curves
which are tips of $D$.
 \item A simple loop of $\mathbb{P}^1$'s.
 \item A tree of $\mathbb{P}^1$'s with a unique branch point which meets three
linear trees defining cyclic quotient singular points at one of their end
points. Further, the absolute values $d_1,d_2,d_3$ of the determinants of the
three trees satisfy $\Sigma 1/d_i=1$. 
 \item A tree of five $\mathbb{P}^1$'s with a unique branch point which
intersects the other four curves transversally in one point each, and such that
the four curves are all (-2) curves.
 \item A smooth elliptic curve.
\end{enumerate}

\end{lemma}

\begin{lemma}[{\cite[Lemma 6.20]{F}}]
Suppose $\ol \kappa(V) \geq 0$. In the case when $N \neq Bk(D)$, one of the
following assertions is true:
\begin{enumerate}
 
 \item There exists a (-1) curve $L$ on $V$ such that $L \cap D= \emptyset$.
 \item There exists a (-1) curve $L$ on $V$ not contained in $D$ such that $L$
       meets a rational twig of $D$ transversally in one point and no other
       curve of $D$.
 \item There exists a (-1) curve $L$ on $V$ which meets a rational twig of $D$ in one point transversally and a tip of a
       rational club in one point transversally and no other point of $D$.\\
       In all these cases $L$ is a component of $N$.
\end{enumerate}
Moreover, $\ol \kappa(V-D-L) = \ol \kappa(V-D)$.
\end{lemma}
Finally, a normal surface singularity $S$ with minimal resolution $\pi: X\to S$ is called minimally 
elliptic if $(K+Z)\cdot E_i=0$ for every irreducible component $E_i$ of the exceptional curve $E$ of 
the resolution; here $Z$ 
is the fundamental 
cycle of the resolution, and $K$ is the canonical divisor of $X$.

\section{Proof of Assertion (1) of Theorem 0.1}
For this part $V$ is affine. We will verify the conjecture for the three cases: $\ol \kappa(V^0) = -\infty ,
0 ,1$. 

\subsection{The case $\ol \kappa(V^0)= -\infty$}

If the smooth surface $V^0$ is affine-ruled (i.e., it contains an open set of
the type $\mathbb{A}^1 \times U_0$) then the $\mathbb{A}^1$-fibration extends
to $V$ (see \cite[\S~1.3.1]{M}). All the
singularities on $V$ are therefore rational singularities \cite[Lemma 2.7]{GM}. Hence $V$ is smooth by Theorem 1.2. If
$V^0$ is not affine ruled then we have the following result of
Miyanishi-Tsunoda: 

\begin{theorem}[{\cite[Theorem 2.5.3, Theorem 2.5.4]{M}}]
 Let $V^0$ be a smooth open algebraic surface defined over an algebraically closed field $k$ of char. $0$ with 
$\ol\kappa(V^0) = -\infty$. Suppose that $V^0$ is not affine ruled. Suppose
furthermore that there exists an open embedding of $V^0$ into a smooth projective
surface $\ol V$ such that
\begin{enumerate}
 
\item $\ol{V}-V^0$ is a reduced effective divisor with simple normal crossings, and
\item if we write $\ol{V} - V^0 = \bigcup C_i$ with components
$C_i$, the intersection matrix $((C_i \cdot C_j))_{1 \leq i,j \leq r}$ is not
negative definite.
\end{enumerate}
Then there exist a Zariski open subset $U$ of $V^0$ and a proper birational
morphism $\phi: U \rightarrow T'$ onto a smooth algebraic surface $T'$ such that

\begin{enumerate}

\item either $U=V^0$ or $V^0-U$ has pure codimension one, and
\item $T'$ is the quotient $\A^2/G-Sing~(\A^2 / G)$, where $G$ is a finite
subgroup of $GL(2,k)$.
\end{enumerate}

\end{theorem}
\nind
First we note that since $D$ supports an ample divisor for a projective compactification for $V$ the intersection 
form on the components of $D$ has one positive eigenvalue. Hence the condition (2)
in Theorem 2.1 is satisfied for $D\cup E$. Thus, if $V^0$ is not affine ruled then it contains an open set 
$U \subset V^0$ isomorphic to $\A^2/G-(0,0)$ because the map $\phi$ is proper and birational and $V$ is affine 
(so that $U$ cannot contain any complete curve). The tangent bundle of $U$ is trivial as $U$ is a subset of $V^0$. 
Hence the tangent bundle of $\A^2/G-(0,0)$ is trivial, but $\Gamma(\A^2/G)$ is positively graded. By Theorem 1.2,
the quotient $\A^2/G$ is smooth, and hence isomorphic to $\A^2$. 
This implies that $\A^2/G$ contains a
cylinderlike open subset. Hence so does $U$. This contradicts the earlier assumption that
$V^0$ is not affine ruled. Therefore, we conclude that $V$ is smooth.

\subsection{The Case $\ol \kappa (V^0)=0$}
In this section we will prove the following:
\begin{lemma}
Let $V$ be an affine surface such that
\begin{enumerate}
\item $V^0$ has a trivial tangent bundle, and
\item $\ol \kappa(V^0)=0$.
\end{enumerate}
Then $V$ is smooth.
\end{lemma}

Let $p_1, p_2,\cdots ,p_r$ be the singular points of $V$. Let $\pi: \ol V \rightarrow V$ be a resolution of 
singularities such that $\pi^{-1}\{p_1,p_2,\cdots ,p_3\}=E$ is an SNC divisor. We can assume without loss of 
generality that any $(-1)$-curve in $E$ meets at least three other components of $E$. Let $X$ be a smooth 
projective completion of $\ol V$ such that $D'= X- \ol V$ is an SNC divisor. Let $D = D' +E$. Clearly $D$ is the 
disjoint union of $D'$ and $E$.\\

Let $P+N=K+D$ be the Zariski-Fujita decomposition of $K+D$, where $P$ is the positive part and $N$ is the 
negative part of the decomposition. Because $\ol \kappa(V^0)=0$, by an important result due to 
Y. Kawamata (Lemma 6.11 of [3]) $P$ is numerically equivalent to 0. Thus $K+D \approx N$. We now state a claim 
whose proof we postpone till the end of this subsection.

{\bf Claim:} {$N=0$}.\\

Assuming the above claim we break the proof of the lemma in two cases
\begin{enumerate}
\item $Bk(D)=N$
\item $Bk(D) \neq N$.
\end{enumerate}

{\bf Case 1: $Bk(D)=N$}\\
The Bark of $D$ is defined as the sum of the Barks of all rational clubs and maximal twigs. By the claim, $Bk(D)=0$, 
hence $D$ can have no twig or rational club. Hence  by Lemma 1.3, every connected component of $D$ is either a 
loop of rational curves or an elliptic curve. It follows that $K+D \approx 0$. Hence $|nK|= \emptyset$ for all 
$n \geq 1$. Thus $X$ has a $\mathbb{P}^1$-fibration. Let $l$ be a general fiber of the $\mathbb{P}^1$-fibration. 
By the arithmetic genus formula, $K \cdot l=-2$. Hence $D \cdot l =2$. The divisor $D$ is a disjoint union of 
$D'$ and $E$. Neither $D'$ nor $E$ can be contained in a fiber of the $\mathbb{P}^1$-fibration because the connected 
components of both are either a loop of rational curves or an elliptic curve.
If $D'$ (respectively, $E$) contains a loop 
of rational curves then $D' \cdot l \geq 2 $ (respectively, $E \cdot l \geq2$) and $E \cdot l \geq 1$
(respectively, $D' \cdot l \geq 1$). 
This contradicts the fact that $D \cdot l=2$.\\
Hence we conclude that  $D'$ and $E$ are smooth elliptic curves and that both $D'$ and $E$ are cross-sections of 
the $\mathbb{P}^1$-fibration. Next we claim that the $\mathbb{P}^1$-fibration is actually a $\mathbb{P}^1$-bundle. 
Indeed, if a fiber $F$ has two components $l_1$ and $l_2$ such that $l_1$ meets $D'$ and $l_2$ meets $E$, 
then $l_2$ must also meet $D'$ as $V$ is affine and contains no complete curves. This is a contradiction 
because $D'$ is a cross section of the $\mathbb{P}^1$-fibration.\\

Thus we can think of the two disjoint cross sections of the $\mathbb{P}^1$-bundle to be $0$ and $\infty$. Because the
two sections of the bundle are disjoint, there is a
trivialization of this bundle with the transition functions
as the M\"obius transforms which leave $0$ and $\infty$ fixed.  The only M\"obius
transforms which leave
$0$ and $\infty$ fixed are those acting by scalar multiplication. These
transforms commute with the natural action of $k^*$ on each of the
trivializations. Thus we have a $k^*$-action on the bundle and hence $V$ is a
positively
graded domain. We know the conjecture to be true for positively graded domains
by
Theorem 1.2. \\

{\bf Case 2: $Bk(D) \neq N$}\\
By Lemma 1.4 we have a $(-1)$-curve $L$ on $\ol V$, not in $D$, which meets $D$ as in Lemma 1.4. We claim that $L$ can 
only meet $D'$. Indeed, If $L$ meets $E$ then it has
to meet $D'$ because $V$ is affine and cannot contain a complete curve. Thus $L$
meets two connected components of $D$. By Lemma 1.4 we know that one
of the components is a contractible rational club, which gives a rational
singularity when contracted. This is not possible by Theorem 1.2. Hence we now know that $L$ can only 
intersect $D'$. We contract $L$ and successively any other curves which become $(-1)$-curves as a result and meet 
the new $D$ as above; this process will stop before all the curves in $D'$ are contracted because the intersection form
on the components of $D'$ has a positive eigenvalue. None of the contracted curves meet $E$ so that $E$ is unchanged 
in this process. This way we ensure that there are no exceptional curves on $\ol V$ meeting $D$ as in Lemma 1.4. Hence 
by Lemma 1.4, we are reduced to the case $Bk(D)=N$. \\
 
We now give a proof of the claim that $N=0$. Suppose that $N \neq 0$. By \cite[\S~6.16]{F}, $N$ is a divisor with 
positive rational coefficients, each coefficient being strictly less than 1. As seen above, $|nK|=\phi$ for all $n>0$, 
so that $X$ has a $\mathbb{P}^1$-fibration $f:X\rightarrow C$. If $N\neq 0$ then by Lemma 1.3 either $D'$ or $E$ 
is a tree of smooth rational curves. This easily implies that $C$ is a rational curve, and hence $X$ is a rational 
surface. Therefore numerical and rational equivalence on $X$ are same. The following
Lemma 2.3 will show that there 
are no rational equivalence relations between the components of $D$, so that the expression of a canonical divisor 
as a linear combination of components of $D$ is unique. We know that $K+D \approx N$. Since the tangent bundle of 
$V^0$ is trivial there is a canonical divisor of $X$ which is supported on $D$. It follows that if 
$N \neq 0$, then the divisor $K$ will have two distinct expressions as a linear combination of the irreducible 
components of $D$. One expression will involve only integral coefficients and the other will be $K \approx N-D$. 
Here $N-D$ has at least one non-integral coefficient as all the coefficients of $N$ are strictly less than 1. Thus $N=0$.  

\begin{lemma}
If $\ol \kappa (V^0)=0$ and the tangent bundle of $V^0$ is trivial, then there
is no non-constant invertible regular function on $V^0$.
\end{lemma}
\begin{proof}
If $u$ is an invertible non-constant regular function on $V^0$, then we have a morphism
$u:V^0
\rightarrow k^*$. By the Stein Factorization theorem there exists a curve $B$
and morphisms $u': V^0 \rightarrow B$ and $\phi: B \rightarrow k^*$ such that
\begin{itemize}
\item $ \phi \circ u' = u$, and
\item the general fibers of $u'$ are irreducible.
\end{itemize}
The curve $B$ is affine because $k^*$ is affine.  For the morphism
$u'$ we have Kawamata's inequality: $\ol \kappa(V^0) \geq \ol \kappa(B) + \ol
\kappa(F)$ (\cite{Kw1}), where $F$ is a general fiber of $u'$. Thus $0 \geq \ol \kappa(B) + \ol \kappa(F)$. There 
is a dominant morphism $\phi: B \rightarrow k^*$, hence $\ol \kappa(B) \geq 0$. Therefore, $0 \geq
\ol \kappa(B) + \ol \kappa(F) \geq \ol \kappa(F)$. If $\ol \kappa(F)= -\infty$
then $\ol \kappa(V^0)=-\infty$, because by Iitaka's easy addition theorem $\ol
\kappa(V^0) \leq \ol \kappa(F)+1$ \cite[Theorem 10.4]{I}. Hence $\ol \kappa(F) =0$. Thus the
morphism $u' : V^0 \rightarrow B$ is a $k^*$-fibration. Because codimension of $V - V^0
\geq 2$ and $V$ is affine, $u'$ extends to $V$ \cite[Theorem 2.18]{I}.
Thus $V$ has a $k^*$-fibration. Hence the singularities of $V$ will be rational, since it is well-known,
\cite[Lemma 2.7]{GM}, that any connected subtree of a singular fiber of a $\mathbb{P}^1$-fibration on a 
smooth projective surface contracts to a rational singular point. This contradicts Theorem 1.2. 
\end{proof}

This completes the proof of Lemma 2.2.

\subsection{The Case $\ol \kappa(V^0)=1$}
In this subsection we will prove the following.
\begin{lemma}
Let $V$ be an affine surface such that
\begin{enumerate}
\item $V^0$ has a trivial tangent bundle, and
\item $\ol \kappa(V^0)=1$.
\end{enumerate}
Then $V$ is smooth.
\end{lemma}

By a result of Kawamata, \cite[Chapter II, Theorem 2.3]{M81}, there is a $k^*$-fibration on $V^0$. 
Hence we have the following possibilities:

(1) The $k^*$-fibration extends to $V$. \\
In this case by Lemma 2.7 of \cite{GM}, the singularities of $V$ are rational. Then by Theorem 1.2, 
the tangent bundle of $V^0$ cannot be trivial.

(2) The $k^*$-fibration does not extend to $V$.\\

Any base point of the linear pencil in $V$ is at a singularity of $V$ because the base point lies outside 
$V^0$. If a singular point of $V$ is not a base point then it lies inside a fiber of a $\mathbb{P}^1$-fibration. 
By lemma 2.7 of \cite{GM} it is a rational singularity. Hence by Theorem 1.2 the tangent bundle of $V^0$ cannot 
be trivial. Let $\pi: \ol V \rightarrow V$ be a minimal resolution of singularities such that the 
$k^*$-fibration extends to the surface $\ol V$. Let $X$ be a smooth projective completion of $\ol V$. 
Let $E$ be the exceptional divisor of $\pi$. $E$ is not contained in a fiber of the $\mathbb{P}^1$-fibration 
on $X$. The divisor $D' = X- \ol V$ is not contained in a fiber of the $\mathbb{P}^1$-fibration on $X$ because 
its intersection form has a positive eigenvalue. Since $V^0$ has a $k^*$-fibration, both $E$ and $D'$ have
exactly one component horizontal to the $\mathbb{P}^1$-fibration on $X$.\\

Any singular fiber of the $\mathbb{P}^1$-fibration on $X$ will contain components of $E$ or $D'$. Indeed, if $F$ is 
a singular fiber such that no component of $F$ is in $E$, then let $l_1$ be the components of $F$ meeting $E$. 
Since $V$ is affine, $l_1$ also meets $D'$. If there is no component of $D'$ in $F$, then $l_1$ meets the 
horizontal component of $D'$. If $l_2$ is another component of $F$, then $l_2$ also meets the horizontal components 
of $D'$. This means that the horizontal component of $D'$ is not a cross-section, which is a contradiction. 
Thus $l_1$ is the only component of $F$, i.e., $F$ is a smooth fiber.\\
Next we claim that all the singular fibers of the $\mathbb{P}^1$-fibration on $X$ are linear chains of 
rational curves. Every singular fiber $F$ contains a unique rational curve which intersects 
both $D'$ and $E$. Indeed, this follows from the observations that $D',E$ contain cross-sections of 
the $\mathbb{P}^1$-fibration 
and $F$ is a tree of rational curves. We may assume that there are no (-1) curves in $F$ which intersect 
only $D'$, or only $E$ as such curves can be contracted without affecting the triviality of the tangent 
bundle of $V^0$. Also, there are no (-1) curves in $D'$ or $E$ which are in $F$, as they too can be contracted. 
Using Lemma 1.1 we deduce that the unique rational curve in $F$ which meets both $D'$ and $E$ has self-intersection 
$-1$. By Lemma 7.6 of \cite{F}, the fiber $F$ is linear.

The components of this
fiber can be contracted successively to give a $\mathds{P}^1$-bundle with two
sections which we can call $0$ and $\infty$. On each trivialization of this
bundle we have a $k^*$-action. Since the sections do not intersect after the
trivializations have been glued together the transition functions
are multiplication by a scalar. These commute with the multiplication action of
$k^*$, proving that there is a $k^*$-action on the $\mathds{P}^1$-bundle that
leaves the two sections invariant. We can now blow up at the two sections to
reverse the blowing down process. Any point of intersection of two irreducible components of the new singular
fiber is fixed for the $k^*$-action. Hence the action extends to the blow up $\ol V$. This
proves that the surface $V$ is a graded domain, in which case we know the
conjecture to be true by Theorem 1.2. This completes the proof of Lemma 2.4.

\section{Proof of Assertion (2) of Theorem 0.1}
\begin{lemma}
Let $V$ be a normal projective surface. If the tangent bundle of $V^0$
is trivial, then $\ol \kappa(V^0) \leq 0$.
\end{lemma}

\begin{proof}
Let $p_1,p_2,\cdots ,p_r$ be the singular points of $V$. Let $\pi: M \rightarrow V$
be a minimal resolution with the exceptional divisor $E =
\pi^{-1}\{p_1,\cdots ,p_r\}$. Let $K$ be the canonical divisor of $M$.\\
We claim that $K \cdot E_i \geq 0$ for all components $E_i$ of $E$.
Suppose $K \cdot E_i \textless 0$ for some $E_i$. By the adjunction formula
$(E_i^2+E_i \cdot K)/2 +1 \geq 0$; this forces $E_i^2=-1$. This contradicts the assumption that the resolution
$\pi$ is minimal.  Thus $K \cdot E_i \geq 0$. Since $E$ has a negative-definite
intersection form, we deduce that $K$ is a negative divisor supported on $E$. Thus $-K \geq 0$ and $-K
\cdot E_i \leq 0$. If $K$ is equivalent to the zero divisor, then $p$ is a
rational double point; in that case by Theorem 1.2 the tangent bundle of $V^0$ cannot be
trivial. Hence $-K \geq Z$, where $Z$ is the fundamental
cycle of $E$. Therefore, the support of $K$ is the union of all the components of $E$.
We perform further blowups on $M$,
$$
\pi_1:M_1 \rightarrow M\, ,
$$
so that $E_{M_1}:=\pi_1^{-1}(E)$ is an SNC divisor and $\pi_1$ is minimal with
respect to this property. Note that $K_{M_1}$ is negative,
and the support of $K_{M_1}$ contains $|\pi^{-1}(E)|$. Hence $K_{M_1}+E_{M_1} \leq 0$. If $K_{M_1}+E_{M_1} \textless 0$ then
$\ol \kappa(V^0) = -\infty$. If $K_{M_1} + E_{M_1}=0$ then $\ol \kappa(V^0) =
0$.
\end{proof}

In view of the above lemma it is enough to verify the conjecture for the cases
$\ol \kappa(V^0)=0$ and $\ol \kappa(V^0)= -\infty$.\\ 

We also have the following:

\begin{lemma}
Let $V$ be a normal projective surface with $\ol \kappa(V^0) \leq 0$.
If the tangent bundle of $V^0$ be trivial, then $V$ has at most two
singularities.
\end{lemma}

\begin{proof}
The canonical bundle of $V^0$ is trivial because its tangent bundle is
trivial. Thus a canonical divisor of $\ol V$ is supported  on $E =
\pi^{-1}\{p_1,p_2,\cdots ,p_n\}$, where $p_i$ are the singular points of $V$ and
$$\pi:
\ol V \rightarrow V$$ is a minimal resolution. We know that the
support of $K$ contains $E$ since $K \textless 0$ and $-K\geq Z$. Thus 
$\kappa(\ol V) = -\infty$. Hence we have a $\mathbb{P}^1$-fibration on $\ol V$. Let $l$ be a general fiber of the
$\mathbb{P}^1$-fibration. By the adjunction formula $K \cdot l = -2$. We claim that $E$ has at least one
irreducible component which is not contained in any fiber. For, if this is false then we would have $K\cdot l=0$ since 
$K$ is supported on $E$. But $K\cdot l=-2$. This contradiction proves the claim. Now $E$ has components which
are not contained in any fiber. Hence there can be at most two such components of $E$ which are horizontal. Any
connected component of $E$ which lies entirely in a fiber contracts to a rational double
point. Such a singularity cannot exist because the tangent bundle of $V_0$ is
trivial. Hence every connected component of $E$ has a component
which is horizontal to the $\mathbb{P}^1$-fibration. It follows that the number of singularities is at most two.
\end{proof}

Since the tangent bundle of $V^0$ is trivial, $V^0$ cannot contain a
smooth projective rational curve. Indeed, if $\Delta$ is a smooth projective
rational curve in $V^0$ then it cannot have positive or zero self-intersection as it would violate the
adjunction formula. If the self-intersection is negative then by adjunction formula this self-intersection is 
an even number. In this case the curve is
contractible to a rational singularity. This is not possible by Theorem 1.2.

We now state a proposition proved in \cite[\S~6.2]{W} about curves with
negative self-intersection. The proof of this result assumes that $k=\C$. Using complete local rings and a suitable
application of $``$Lefschetz Principle'' Lemma 3.4 below, which uses this result, is valid for any 
algebraically closed field of characteristic $0$.

\begin{theorem}
 Suppose $p$ is an isolated singularity of $V$ and $\pi :M \rightarrow V$ is the
minimal resolution $\pi^{-1}(p)=E$ is a non-singular curve of genus $g$ and $E
\cdot E \textless 4-4g$. Let $N$ be the normal bundle of $E$ in $M$. Then there
is
a neighbourhood (in the complex topology) $U$ of $E$ in $M$ and a neighbourhood
$U'$ of $E$ in the total space of $N$ and an analytic isomorphism $\sigma: U
\rightarrow U'$, such that $\sigma$ is identity on $E$.
\end{theorem}

Using Theorem 3.3 we strengthen Lemma 3.2 as follows.

\begin{lemma}
With the hypothesis as in Lemma 3.2, the surface $V$ can have at most
one singularity.
\end{lemma}

\begin{proof}
 Suppose that $p_1$ and $p_2$ are two distinct singular points. Let $E = E_1
\cup E_2$. From the arguments in the proof of Lemma 3.2 it follows that each $E_i$ contains a unique component which is a 
cross-section for the $\mathbb{P}^1$-fibration on $\ol X$. It follows for $i=1,2$ that the part of $E_i$ 
contained in a singular fiber is connected.\\ 

{\it Claim 1}. $E_1$ and $E_2$ do not contain any components that are vertical to the $\mathbb{P}^1$-fibration. 

For suppose that $F$ is a singular fiber of the $\mathbb{P}^1$-fibration that contains 
a component of $E_1$ (a similar proof works for $E_2$). Then $F$ is reducible and every component of F has 
negative self-intersection. Now $E_1$ and $E_2$ are disjoint and the
fibers of the fibration are connected. The intersection form on a fiber cannot be negative definite. Hence there 
exists a smooth rational curve $h$, not contained in $E_1,E_2$, such that $h \subset F$ and 
$h \cdot E_1 \neq 0$ or $h \cdot E_2 \neq 0$. There cannot exist a complete smooth rational curve on $V^0$. 
Hence any component $h'$ of $F$ must meet at least one of $E_1,E_2$. 

For the proof of Claim 1 we need the following.\\

{\it Claim 2.} $F$ contains at most two rational curves $h_1$ and $h_2$ that are not contained in $E$. 

For, if this is not true then let $h_1,h_2,h_3,\cdots$ be components of $F$ which are not contained in $E_1\cup E_2$. 
Each $h_i$ is a (-1) curve. For, let $h_1$ meet a component of $E_1$ whose coefficient
in $K$ is $-a$ where $a \textgreater 0$. Then by the adjunction formula we deduce that $h_1 \cdot h_1 =-1$ 
(since $h_1.\cdot h_1<0$). This also proves that $h_i$ cannot meet both $E_1,E_2$. Finally, by connectedness of $F$ 
it follows that there exist $h_i\neq h_j$ such that they meet. But then $h_i\cup h_j$ is a full
singular fiber since $(h_i+h_j)\cdot (h_i+h_j)=0$. This proves the Claim 2, and also Claim 1 when $F$ contains exactly two 
rational curves which lie outside of $E$.\\

Finally, suppose that there is exactly one component $h$ outside $E$ and inside $F$. Again, connectedness of $F$ 
implies that $h$ meets both $E_1,E_2$. As above, by adjunction formula $h \cdot h =0$. This is because 
$K\cdot h=-2$. But then $h$ is a full fiber.

This proves Claim 1 that there are no vertical components in $E$.\\

Because $E$ consists of two disjoint sections both with negative self intersection, the
$\mathbb{P}^1$-fibration is not a bundle. This is because it is well-known that on a $\mathbb{P}^1$-bundle there is at most
one curve which has a negative self-intersection. By an argument similar to
the one above, we see that every singular fiber is a union of two rational
curves with self intersection (-1), meeting each other transversally. Let
$E_1^2=-m$ and $E_2^2=-n$. For every singular fiber we contract the (-1) curve
meeting $E_2$. In the minimal model $E_1^2=-m$ and $\widetilde{E_2}^2=m$  where $\widetilde{E_2}$ is the image 
of $E_2$. This can be seen by writing $\widetilde E_2$ as $\theta_1E_1 + \theta_2 f$.

We claim that the genus of $E_1$ (and hence also  of $E_2$) is at most 1. We know that $K+E \leq 0$, where $E$ is
supported on $E_1 \cup E_2$. If $K+E=0$ then $K= -E_1-E_2$. As $E_1 \cdot
E_2=0$, by the adjunction formula $E_1$ is an elliptic curve. If $K+E \textless
0$ then $K= -aE_1-bE_2$, where $a$,$b$ are positive integers. By the adjunction
formula we have: $(\widetilde {E_2}^2+ E_2(-aE_1-b\widetilde{E_2}))/2 +1 \geq 0$, $\widetilde{E_2}^2 \textgreater 0$. 
Thus the genus of $E_2$ is at most 1.\\

If the genus of $E_1$ and $E_2$ is 0,
then the two singularities are rational which is not possible by Theorem 1.2.\\ 

Finally, $E_1$ and $E_2$ cannot have genus 1. Indeed, if $E_1$ is a smooth elliptic curve
then by Theorem 3.3 we
can embed a neighbourhood $U$ of $E_1$ in the normal bundle of $E_1$. The normal
bundle of $E_1$ has a $k^*$-action because its completion has two disjoint
sections. The tangent bundle of $U-E_1$ inside the normal bundle is trivial
hence
the tangent bundle of the complement of $E_1$ in the normal bundle is trivial.
But the completion of the normal bundle has two disjoint sections and hence  by
an argument similar to that in Lemma 2.2 the completion has a $k^*$-action which leaves the two sections fixed, hence 
the coordinate ring of the normal bundle is a positively graded domain. Now $E_1$ is a
curve with
negative self-intersection in the normal bundle which can be contracted to a singularity. By Theorem
1.2 (1), the tangent bundle of the complement of $E_1$ cannot be trivial, proving the
claim.\\
\end{proof}

\section{Proof of Assertion (3) of Theorem 0.1}
In this section we will prove the following
\begin{lemma}
Let $V$ be a projective surface over $k$ such that
\begin{enumerate}
\item $V^0$ has a trivial tangent bundle, and
\item $\ol \kappa(V^0)=0$.
\end{enumerate}
Then $V$ is smooth.
\end{lemma}
By the results in the previous section we know that if $V$ is singular then there is a unique singular point,
which we will denote by $p$, 
of $V$. Let $\ol V \rightarrow V$ be a resolution of the singularity $p$ such that letting $\pi^{-1}(p)=E$ to be the 
exceptional divisor $E$ is SNC.\\

To prove that $V$ is smooth we have to consider the following two cases:\\
\begin{enumerate}
\item $Bk(E)=N$
\item $Bk(E) \neq N$
\end{enumerate}

where $N$ is the negative part of $K+D$.\\
In the case when $Bk(E) \neq N$ there is a (-1)-curve $L$ on $V$, not contained in $E$, such that $L$ meets a unique 
rational twig of $E$ transversally in one point. We contract $L$ and any other (-1) curve that arises in the image of
$\ol V$ as a result and which plays the role of $L$. Continuing this process we are reduced to the case when $Bk(E)=N$.\\
We may thus assume that $Bk(E)=N$.\\

We now state a claim whose proof we postpone till the end of this section.

{\bf Claim:} $E$ is either a smooth elliptic curve or a loop of rational curves.\\

We will show that even the above two possibilities can be ruled out, thus proving that $V$ is smooth.\\

$E$ cannot be an elliptic curve. For otherwise, by Theorem 3.3, a small Euclidean neighborhood of $E$ in $\ol V$ is
biholomorphic to a neighborhood of $E$ in the total space $P$ of the normal bundle of $E$ in $\ol N$. 
Then $P-E$ has trivial tangent bundle. After contracting the zero-section of the normal bundle we get a 
positively graded affine surface
such that its smooth locus has a trivial tangent bundle. This contradicts Theorem 1.2 (1).  

If $E$ is a loop of rational curves then it contracts to a {\it cuspidal singularity}. It was 
proved by Neumann and Wahl, \cite[Proposition 4.1]{neu}, that given a germ of cuspidal singularity $(W,q)$ there exists a 
germ of a hypersurface cuspidal singularity $(\widetilde W,\widetilde Q)$ with a finite map 
$(\widetilde W,\widetilde Q)\rightarrow (W,q)$
which is ramified only at the singular point. Thus we have a hypersurface singularity 
whose smooth locus has trivial tangent bundle. By Theorem 1.2, the surface is smooth. This is a contradiction.

We now finish this section by proving the Claim.

\begin{lemma}
Let $V$ be a normal projective surface. Assume that the tangent bundle of $V^0$
is trivial. Let $\pi: M \rightarrow V$ be a resolution of
singularities such that the exceptional divisor $E$ is SNC. Further assume that
$N = Bk(E)$, where $N$ is the negative part of the Zariski-Fujita
decomposition of $K+E$. Then $Bk(E)= \emptyset$.
\end{lemma}

\begin{proof}
Since $\ol \kappa(V^0)=0$, the positive part of the Zariski-Fujita
decomposition is zero. Thus
\begin{equation}\label{eqK}
K \approx N - E\, .
\end{equation}
If $N$ is non-zero, then all the
coefficients of $N$ are positive and less than 1. The coefficients of $E$ and
$K$ are integers. The components of $E$ are also the support of $K$
 because the tangent bundle of $V^0$ is trivial. Also, because $E$ is negative
definite, the 
components of $E$ are numerically independent. But the
coefficients on the left hand side of \eqref{eqK}
are integers and on the right hand side of \eqref{eqK} are fractions.
This is a contradiction.
\end{proof}

Lemma 4.1 and Lemma 1.3 imply that
when $Bk(E)=N$, each component of $E$ is either a smooth elliptic curve or a
loop of rational curves. Because for all the other possibilities in Lemma 1.3 $Bk$ is
non-empty.\\ 

\section{Proof of assertions 4, 5 and 6}
By Assertion (2), the surface $V$ can have at most one singularity. Let $p$ be the unique
singular point of $V$. Let $\pi:W \rightarrow V$ be a resolution of the singular
point $p$ such that $E= \pi^{-1}(p)$ is an SNC divisor. Thus $V^0 = W-E$. By
hypothesis, $\ol \kappa (W-E)= -\infty$. Hence we have a $\mathbb{P}^1$-fibration $f: W
\rightarrow C$, \cite[Chapter I, \S~3.13]{M81}, and there exists an $\mathbb{A}^1$-fibration on $V^0$ 
because $\ol \kappa(V^0)= -\infty$. Let $S$ denote the irreducible component of $E$ which is a cross-section to $f$.

We prove the following.
\begin{lemma}
If $C$ is a rational curve then $V$ is smooth.
\end{lemma}
\begin{proof}
If $g(C)=0$, then the irregularity $q$ of $W$ is 0. Let $P_g$ denote the geometric genus of the singularity $p$. 
By a theorem of Umezu, \cite[Theorem 1]{U}, $P_g(p)=1$. The singularity $p$ is Gorenstein, hence $p$ is minimally 
elliptic. It is proved by Okuma, \cite[Theorem 4.3]{O}, that if $(X,o)$ is a minimally elliptic singularity, then 
there exists a finite cover $(Y,o') \rightarrow (X,o)$ ramified precisely at the singularity $o$ such that $(Y,o)$ 
is a complete intersection singularity. Because the covering is unramified outside $o$, the tangent bundle of 
$Y-o$ is also trivial. The conjecture is known to be true for complete intersections (Theorem 1.2). 
This completes the proof.
\end{proof}

\nind
{\bf The case $g(C)\geq 2$}.\\

Now assume that $g(C) \geq 2$. Let $F$ be a singular fiber of the fibration $f$. Assume that $F$ has a feather 
with multiplicity 1. Such a feather is called a reduced feather. By the negativity of the canonical divisor and 
the adjunction formula we see that every feather is a (-1) curve. On such a singular fiber with a reduced feather 
we perform a \emph{contraction process} as follows: Because the feather is reduced there is another (-1) curve 
in $F$. We contract this (-1) curve distinct from the reduced 
feather and continue contracting any (-1) curves that arise till we are left with only the feather as the 
unique irreducible component of the fiber. Note that when a $(-1)$-curve different from the reduced feather is
contracted, we are in effect removing from $V^0$ some curve. Clearly this open subset of $V^0$ also has a trivial
tangent bundle.\\

We consider two cases:
\begin{enumerate}
\item Every singular fiber has a reduced feather.
\item There is a singular fiber $F$ such that all feathers of $F$ are multiple.
\end{enumerate} 

We first consider the first case.\\

We perform the \emph{contraction process} on every singular fiber keeping the reduced feather till the end and 
denote the new ruled surface by the same symbols: $f: W \rightarrow C$. As every feather lies outside the support of the 
canonical divisor, the canonical divisor $K_W$ is now supported on the section $S$ of the ruling. By the 
general theory of ruled surfaces $K_W ^2 = 8(1-g(C))$. Hence $S$ has negative self-intersection. Since the complement 
of $S$ has trivial tangent bundle by the above remark, we only have to prove that $V$ is smooth when $E=S$.

\begin{lemma}
If $E=S$, then $V$ is smooth.
\end{lemma}

{\bf Remark.} The proof below works even when genus $C=1$ if $S^2<0$ In the following arguments we sometimes assume that
the field $k=\C$. An application of Lefschetz Principle will enable us to deduce the result for an algebraically 
closed field $k$ of char. $0$.\\
 
\subsection{Proof of Lemma 5.2}

Let $C$ be an irreducible smooth projective curve of genus $g$,
with $g\, \geq\,1$. Let $V$ be a vector bundle of rank two over $C$. Let
$$
f\, :\, W\,:=\, {\mathbb P}(V)\,\longrightarrow\, C
$$
be the ${\mathbb P}^1$--bundle parametrizing lines in the fibers
of $V$. Let
\begin{equation}\label{f2}
T_f\, \subset\, TW
\end{equation}
be the line subbundle given by the kernel of the differential
$df\, :\, T_W\,\longrightarrow\, f^*T_C$.

Line subbundles of $V$ are in bijection with the sections of $f$.

We will prove the result in several steps.\\

{\bf Step 1.}\\

Let $L\,\subset\, V$ be a line subbundle, and let
$$
\sigma\, :\, C\, \longrightarrow\, W
$$
be the section corresponding to $L$. Let $E:=\sigma(C)$ be the image. Let $Q:=V/L$.
The normal bundle $N_{E/W}\cong {\mathcal O}(E)|_E\cong L^*\otimes Q$.
In particular,
\begin{equation}\label{sigC}
(\sigma(C))^2 \, =\,
{\rm degree}(L^*\otimes Q)\, .
\end{equation}
This is proved in (\cite{Ma}, Proof of Lemma 1.15).

We assume that there is a section
$$
\sigma\ :\, C\,\longrightarrow\, W
$$
of $f$ such that
\begin{equation}\label{f1}
(\sigma(C))^2\, <\, 0\, .
\end{equation}
Let
\begin{equation}\label{e0}
E\,:=\, \sigma(C)\, \subset\, W
\end{equation}
be the divisor in \eqref{f1}. Note that $E$ is identified with $C$ using the map $\sigma$.

Let
$$
L\, \subset\, V
$$
be the line subbundle corresponding to the section $\sigma$ in \eqref{f1}. We have a short
exact sequence of
vector bundles
\begin{equation}\label{e2}
0\,\longrightarrow\, L\,\longrightarrow\,  V\,\longrightarrow\, Q\,:=\,
V/L \,\longrightarrow\, 0\, .
\end{equation}

{}From \eqref{sigC} and \eqref{f1} we have  $degree~L> degree\, Q$.\\

{\bf Step 2.}\\

There is an isomorphism $T_f|{E}\cong N$.

To see this, we observe that $\bigwedge^2~ T_{W}|_{E}\cong T_E\otimes N$.
Similarly, $\bigwedge^2~T_W|{E}\cong T_f\otimes T_E$. Hence $T_f|{E}\cong N$. 

Let $K_W$ be the canonical line bundle of $W$.
The complement $W\setminus E$ will be denoted by $W_0$.\\

{\bf Step 3.}\\

Assume that the restriction of $K_W$ to $W_0$ is algebraically trivial. Then
$$
T_E\cong L^*\otimes Q\cong N$$

A canonical divisor of $W$ is supported on $E$. Since $K_W.F=-2$ for a fiber $F$ of $f$ we get
$K_W\sim -2E$. By adjunction formula, $K_E\cong (K_W\otimes {\mathcal O}(E))_E$. Thus, 
$K_E\cong {\mathcal O}(-E)|_E$. By taking the dual we get the result by Step 1.

\begin{remark}\label{rem1}
{\rm Assume that the restriction of $K_W$ to $W_0$ is trivial.
Since $T_C\,=\, L^*\otimes Q$ by Step 1 it
follows that ${\rm degree}(T_C)\, <\, 0$. Consequently, $g\, \geq\, 2$.}
\end{remark}

{\bf Step 4.}\\

Assume that $E^2<0$. Then there are no nonconstant regular
functions on $W_0$.

For this we assume that $k=\C$. By Grauert's theorem $E$ can be contracted to a normal singular point $p$ on a 
compact complex surface $W'$. Then $W'-p$ is biholomorphic to $W_0$. Any regular function on $W_0$ extends to $W'$ 
by Hartog's theorem, hence it has to be a constant.\\

\nind
{\bf Step 5.}\\

This is the most crucial step in the proof.

For the line bundle $T_f$ 
\begin{equation}\label{eqinq}
\dim H^0(W,\, T_f)\, \geq\, g\, \geq\, 2\, .
\end{equation}

Consider the restrictions of $T_f$ and ${\mathcal O}_W(2)$
to any fiber of $f$. Both these restrictions are of degree two. Therefore, by the
see--saw theorem, there is a line bundle $\zeta$ on $C$ such that
$$
T_f\,=\, {\mathcal O}_W(2)\otimes f^*\zeta\, .
$$
We noted in Step 2 that
$(T_f)|_{E}\,=\, L^*\otimes Q$. The restriction of ${\mathcal O}_W(2)$
to $E$ is $Q^{\otimes 2}$ after we identify $C$ with $E$ \cite[Chapter V, Proposition 2.6]{Ha}, 
Therefore, $\zeta\, =\, L^*\otimes Q^*$. So,
\begin{equation}\label{f3}
T_f\,=\, {\mathcal O}_W(2)\otimes f^*(L\otimes Q)^*\, .
\end{equation}
By the projection formula,
$$
f_*T_f\,=\, f_*({\mathcal O}_W(2)\otimes (L\otimes Q)^*)\,=\,
\text{Sym}^2(V)\otimes (L\otimes Q)^*\, .
$$
The line bundle $L^2$ is a subbundle of $\text{Sym}^2(V)$. Therefore,
$$
K_C\,=\, L\otimes Q^*\,=\, L^2\otimes (L\otimes Q)^*\, \subset\,
\text{Sym}^2(V)\otimes (L\otimes Q)^*\,=\, f_*T_f\, ,
$$
where $K_C$ is the canonical line bundle of $C$. This implies that
$$
H^0(C,\, K_C)\, \subset\, H^0(C,\, f_*T_f)\,=\, H^0(W,\, T_f)\, .
$$
But $\dim H^0(C,\, K_C)\, =\, g\, \geq\, 2$ (see, Remark after Step 2).
This completes the proof of \eqref{eqinq}.\\

{\bf Step 6.}\\

The tangent bundle $T_{W_0}$ is not trivial.

Assume that $T_{W_0}$ is trivial. Then the canonical line bundle
$K_{W_0}$ is also trivial. Hence $T_C\,=\, L^*\otimes Q$ by Step 3, and
also we have $g\, \geq\, 2$.

Since $T_{W_0}$ is the trivial vector bundle, it is generated
by its global sections. Since the rank of $T_{W_0}$ is two, from Step 4 we
know that
\begin{equation}\label{g1}
\dim H^0(W_0,\, TW_0)\,=\, 2\, .
\end{equation}
On the other hand, from Step 4 we have 
\begin{equation}\label{g2}
\dim H^0(W_0,\, (T_f)|_{W_0}) \,\geq\, \dim H^0(W,\, T_f)\, \geq\, 2\, .
\end{equation}
Since $(T_f)|_{W_0} \subset\, T_{W_0}$, we have
$H^0(W_0,\, (T_f)|_{W_0})\, \subset\, H^0(W_0,\,T|{W_0})$. Therefore, from
we conclude that
$$
H^0(W_0,\, (T_f)|_{W_0})\,=\, H^0(W_0,\, T|{W_0})\, .
$$
But this contradicts the earlier observation that $T|_{W_0}$ is generated by its
global section. In view of this contraction, the proof is complete for Case 1.\\

\nind
{\bf Remark.} Suppose that $C$ is an elliptic curve and $V$ an indecomposable rank $2$ vector bundle over 
$C$ with invariant $e=0$. Let $f:W:={\mathbb P}(V)\rightarrow C$ be the corresponding ${\mathbb P}^1$-bundle. There is a 
cross-section $E$ of $f$ such that $E^2=0$. It can be shown that the tangent bundle of the complement $W-E$ is 
trivial, $W-E$ has no non-constant regular functions, and $H^1(W-E,\Omega^1)=(0)$. For the last assertion, 
see \cite{MK}.\\  

Finally, we deal with case 2.\\

Now there is a singular fiber $F$ of $f$ such that all the feathers of the 
fiber are multiple. Let $F_1, F_2, \cdots , F_n$ be singular fibers such that no feather is reduced. Now 
contract all but one multiple feather in each $F_i$ so that the multiple feather is the only (-1) curve 
in $F_i$. Let $m_1, m_2,\cdots , m_3$ be the multiplicities of the unique feather of each singular fiber. By the 
solution of Fenchel's Conjecture due to Nielsen-Bundagaard and R. Fox \cite{Fo} (see also, \cite{C}), 
there exists a curve $C'$ and a Galois map 
$g: C' \rightarrow C$ which is ramified precisely at $f(F_i)$, $1 \leq i \leq n$ with ramification index $m_i$. The 
normalized fiber product $W \times_C C'$ is again a $\mathbb{P}^1$-fibration such that each singular fiber has at 
least one reduced feather. We are now reduced to Case 1.\\

\nind
Assertion (6) of Theorem 0.1 can be proved in a similar way. Since $C$ is an elliptic curve, we can construct a finite
Galois cover $C'\to C$ with prescribed ramification over only the points such that corresponding fiber has only 
non-reduced feathers and take the fiber product $W\times_CC'$ (which will then be a finite etale cover of $W$), etc.\\

\vspace{5mm}
Indranil Biswas, School of Mathematics, Tata Institute of Fundamental Research, Homi-Bhabha Road,
Mumbai 400005, India.\\
E-mail: indranil@math.tifr.res.in\\

R.V. Gurjar, School of Mathematics, Tata Institute of Fundamental Research, Homi-Bhabha Road,
Mumbai 400005, India.\\
E-mail: gurjar@math.tifr.res.in\\

\nind
Sagar U. Kolte,  School of Mathematics, Tata Institute of Fundamental Research, Homi-Bhabha Road,
Mumbai 400005, India.\\
E-mail: sagar@math.tifr.res.in\\

Current address: School of Mathematics, Korea Institute for Advanced Study, Seoul 130-722, Korea.\\
sagar.kolte@gmail.com\\


\begin{thebibliography}{99}
\bibitem{becker}
J. Becker. \emph{Higher derivations and integral closure.},  Amer. J. Math.,
{\bf 100} (1978), (3):  495-521.

\bibitem{C}
T.C. Chau, \emph{A note concerning Fox's paper on Fenchel's conjecture}, Proc. Amer. Math. Soc. {\bf 88} (1983), 584-586.

\bibitem{D}
S. Druel, THE ZARISKI-LIPMAN CONJECTURE FOR LOG CANONICAL SPACES, arXiv:1301.5910v1

\bibitem{flenner}
H. Flenner. \emph{Extendability of differential forms on nonisolated
singularities.},  Invent. Math., {\bf 94} (1988), (2):  317-326.

\bibitem{F}
T. Fujita, \emph{On the topology of non-complete algebraic surfaces}, Journal of
the Fac. of Sci., The University of Tokyo. Sect. 1 A, Mathematics,
{\bf 29} (1982), (3),  503-566.

\bibitem{Fo}
R.H. Fox, \emph{On Fenchel's conjecture about F-groups}, Math. Tidsskr. {\bf B} (1952), 61-65.

\bibitem{G}
P. Graf, An optimal extension theorem for $1$-forms and the Lipman-Zariski conjecture, arXiv: 1301.7315

\bibitem{GM}
R.V. Gurjar, M. Miyanishi, \emph{On the Jacobian Conjecture for
$\mathbb{Q}$-Homology Planes}, J. reine angew. Math, {\bf 516} (1999), 115-132.

\bibitem{Ha}
R. Hartshorne, Algebraic Geometry, Graduate Texts in Mathematics no. 52, Springer.

\bibitem{H}
M. Hochster, \emph{The Zariski-Lipman conjecture in the graded case}, Journal of
Algebra, {\bf 47} Issue 2, (1977),  411-424. 

\bibitem{I}
S. Iitaka, \emph{Algebraic Geometry, An introduction to the birational geometry
of Algebraic Varieties}, Springer (1971).

\bibitem{K}
R. K\"{a}llstr\"{o}m, \emph{The Zariski-Lipman conjecture for complete
intersections},
Journal of Algebra, {\bf 337}, Issue 1, (2011),  169-180. 

\bibitem{Kw1}
Y. Kawamata, Addition formula of logarithmic Kodaira dimensions for morphisms 
of relative dimension one, Proc. Internat. Symp. on algebraic geometry at Kyoto (1977),
201-217, Kinokuniya, 1978.

\bibitem{Kw2}
Y. Kawamata, \emph{On the classification of noncomplete algebraic
surfaces}, Lecture Notes in Math. {\bf 732}, Springer (1979).

\bibitem{L77}
H. Laufer, \emph{On Minimally Elliptic Singularities}, American Journal of
Mathemmatics, {\bf 99}, No. 6,  (1977),  1257-1295.

\bibitem{Li}
J. Lipman, \emph{Free Derivation Modules on Algebraic Varieties}, American
Journal of Mathematics, {\bf 87}, No. 4 (1965),  874-898. 

\bibitem{M}
M. Miyanishi, \emph{Open Algebraic Surfaces}, CRM Monograph Series, {\bf 12},
(2001).

\bibitem{M81}
M. Miyanishi, \emph{Theory of Non-complete Algebraic Surfaces}, Lecture Notes in
Math. {\bf 857}, Springer, (1981).

\bibitem {Ma}
M. Maruyama, \emph{On Classification of Ruled Surfaces}, Kyoto Univ. Lectures in Mathematics 3, Kinokuniya, Tokyo (1970).

\bibitem{MK}
N. Mohan Kumar, Affine-like surfaces, J. Alg. Geometry, {\bf 2} (1993) no. 4, 689-703.

\bibitem{neu}
Walter D. Neumann and Jonathan Wahl, \emph{Universal abelian covers of quotient-cusps},
Math. Annalen, {\bf 326}, Number 1, (2003), Pages 75-93.

\bibitem{O}
Tomohiro Okuma, \emph{Universal Abelian Covers of Certain Surface Singularities}, Math Ann., {\bf 334}, (2006), 753-773. 

\bibitem{P}
E. Platte, \emph{Ein elementarer Beweis des Zariski-Lipman-Problems f\"ur
graduierte analytische Algebren}, Arch. Math., {\bf 31}, No. 2, (1979), 
143-145.

\bibitem{S}
G. Scheja, U. Storch, \emph{Differentielle Eigenschaften der Lokalisierungen
analytischer Algebren},  Math. Ann., {\bf 197} (1972),  137-170.

\bibitem{steenbrink}
D. van Straten and J. Steenbrink. \emph{Extendability of holomorphic
differential forms near isolated hypersurface singularities.},  Abh. Math. Sem.
Univ. Hamburg, {\bf 55} (1985), 97-110.

\bibitem{U}
Yumiko Umezu, \emph{On Normal Projective Surfaces with Trivial Dualizing Sheaf}, Tokyo J. Math., (1981), Vol 4, No.2, 

\bibitem{W}
P. Wagreich, \emph{Elliptic Singularities Of Surfaces}, American Journal of
Mathematics, {\bf 92}, No. 2 (1970),  419-454. 
\end{thebibliography}
\end{document}